\documentclass[12pt]{article}
\usepackage{latexsym,amsmath,amsthm}
\usepackage{amssymb}
\usepackage{graphicx}
\usepackage{appendix}
\usepackage{tikz}
\usepackage[colorlinks=true, linkcolor=blue]{hyperref}
\everymath{\displaystyle}
\usepackage{setspace}

\newtheorem{thm}{Theorem}[section]

\newtheorem{lemma}[thm]{Lemma}
\newtheorem{conj}[thm]{Conjecture}

\theoremstyle{definition}

\theoremstyle{definition}

\theoremstyle{definition}
\newtheorem{question}[thm]{Question}

\theoremstyle{definition}

\theoremstyle{definition}

\theoremstyle{definition}

\theoremstyle{definition}

\theoremstyle{definition}

\theoremstyle{remark}

\usepackage[margin=1.25in]{geometry}

\usepackage{url,xcolor}

\begin{document}

\title{Trees with non log-concave independent set sequences} 

\author{David Galvin\thanks{Department of Mathematics, University of Notre Dame, Notre Dame IN, United States; dgalvin1@nd.edu. Research supported in part by the Simons Foundation Gift number 854277.
}}

\maketitle

\begin{abstract}
We construct a family of trees with independence numbers going to infinity for which the log-concavity relation for the independent set sequence of a tree $T$ in the family fails at around $\alpha(T)\left(1-1/(16\log \alpha(T))\right)$. Here $\alpha(T)$ is the independence number of $T$. This resolves a conjecture of Kadrawi and Levit.

\vspace{2mm}

\end{abstract}

\section{Introduction}

For a graph $G$ denote by $\alpha(G)$ the size of the largest independent set in $G$ (here an independent set is a set of vertices no two of which are adjacent), and for $t \geq 0$ denote by $i_t(G)$ the number of independent sets in $G$ of cardinality $t$. The sequence $(i_t(G))_{t=0}^{\alpha(G)}$ is the {\it independent set sequence} of $G$.

A sequence $(a_k)_{k=0}^n$ is {\it unimodal} if there is an index $k$ (the {\it mode} of the sequence) with 
$$
a_0 \leq a_1 \leq \cdots \leq a_k \geq a_{k+1} \geq \cdots \geq a_n;
$$
it is {\it log-concave} if $a_k^2 \geq a_{k-1}a_{k+1}$ for all $2 \leq k \leq n-1$, and it is {\it real-rooted} if the polynomial equation $\sum_{k=0}^n a_kx^k =0$ has only real zeroes. Unimodality, log-concavity and real-rootedness are ubiquitous in combinatorics and algebra; see for example the surveys \cite{B2015, B1994, S1989}. It is well known that for a positive sequence $(a_k)_{k=0}^n$ real-rootedness implies log-concavity which in turn implies unimodality, although none of these implications can be reversed in general (see e.g. \cite{B2016}). Unlike its close relative the matching sequence (which is always real-rooted \cite{HL1972}) the independent set sequence of a graph need not even be unimodal; this fact was first observed by Alavi, Malde, Schwenk and Erd\H{o}s \cite{AMSE1987} who also obtained the remarkable result that the independent set sequence of a graph can exhibit arbitrary patterns of rises and falls. 

That observation naturally led them to wonder whether there are families of graphs that have unimodal independent set sequences. One such family is the collection of claw-free graphs (graphs without an induced $K_{1,3}$; these include line graphs). Unimodality of the independent set sequence of claw-free graphs was first established by Hamidoune \cite{H1990} and the stronger real-rootedness property was established by Chudnovsky and Seymour \cite{CS2007}.   

Alavi, Malde, Schwenk and Erd\H{o}s proposed studying two other very natural candidate families --- trees and forests. They asked \cite[Problem 3]{AMSE1987}:
\begin{question} \label{question-AMSE}
For trees (or perhaps forests), is the vertex independence sequence unimodal?
\end{question}
This question, which is often presented in the literature as a conjecture, has received some attention, with numerous partial results (see for example \cite{BES2018, BG2021, BRGGGS2019, B2018, GH2018, H2020, LM2002, WZ2011, Z2007} and references therein) but few very general results. For example while on the one hand it is easy to check that paths and stars (the trees with extreme diameters) have unimodal (even log-concave) independent set sequences --- this provides a possible motivation for Question \ref{question-AMSE} --- on the other hand the status of caterpillars is not known. Here a {\it caterpillar} is a path with some pendant edges dropped from its vertices (not necessarily the same number of pendant edges at each vertex). Caterpillars are a simple common generalization of paths and stars.   

A tempting strengthening of Question \ref{question-AMSE} is to ask, is the independent set sequence of trees always log-concave? Note that log-concavity for trees immediately implies log-concavity for forests, since convolutions of log-concave sequences are log-concave (see e.g. \cite{B2016}). This is not true for unimodality. 

One reason that the strengthening to log-concavity is tempting is because it may be easier to prove. There are more approaches to log-concavity than there are to unimodality, in part because of the uniformity of the log-concavity relations (note that establishing unimodality often requires identifying the location of the mode, usually a non-trivial task). It is also tempting because computations run by Radcliffe \cite{R} showed that all trees on 25 or fewer vertices have log-concave independent set sequences.  However, when Kadrawi, Levit, Yosef, and Mizrachi extended these computations \cite{KLYM2023} they found two trees on 26 vertices that do not have log-concave independent set sequences. Further work by Kadrawi and Levit \cite{KL2023} exhibited infinite families of trees that do not have log-concave independent set sequences. In \cite{RS2025} Ramos and Sun used the AI architecture {\it PatternBoost}, developed by Charton, Ellenberg, Wagner, and Williamson \cite{CEWW2024} to find a plethora of essentially unstructured counterexamples to log-concavity on between 27 and 101 vertices. 

Kadrawi and Levit raised the issue of where along the independent set sequence of a tree log-concavity can break down. Say that log-concavity of a positive sequence $(a_i)_{i=0}^n$ is {\it broken at $k$} if $a_k^2 < a_{k-1}a_{k+1}$. A feature of the two examples found by Kadrawi, Levit, Yosef, and Mizrachi, and of the infinite families of examples constructed by Kadrawi and Levit, is that log-concavity of the independent set sequence is broken only at the last possible place, namely at $\alpha(T)-1$ (where $T$ is the tree under consideration). But Kadrawi and Levit also gave an ad hoc example of a tree $T$ whose independent set sequence is not log-concave, and for which log-concavity is broken slightly earlier, at $\alpha(T)-2$. They made the following natural conjecture:
\begin{conj} \label{conj-KL}
For every $\ell \geq 1$ there is a tree $T=T(\ell)$ for which log-concavity of the independent set sequence is broken at $\alpha(T)-\ell$.  
\end{conj}
In other words, log-cancavity can be broken arbitrarily far from the end of the independent set sequence. 

Here we confirm this conjecture, and show that log-concavity can be broken at a distance around $\alpha(T)/(16\log(\alpha(T)))$ from $\alpha(T)$ (the constant $16$ here is certainly not optimal).
\begin{thm} \label{thm-Theta(root(alpha))}
For every sufficiently large $t$ there is a tree $T_t$ with $\alpha(T_t) = (1+t)[2^{t/16}]$, for which log-concavity of the independent set sequence is broken at $t[2^{t/16}]+2$.
\end{thm}
Observe that
$$
\frac{(1+t)[2^{t/16}]}{16\log\left((1+t)[2^{t/16}]\right)} = (1+o(1))[2^{t/16}]
$$
justifying the assertion made immediate before the statement of Theorem \ref{thm-Theta(root(alpha))}.  
(Here and throughout all asymptotic ($o(\cdot)$) statements are to be interpreted as $t \rightarrow \infty$.)

Our construction of $T_t$ is inspired by the constructions of Kadrawi and Levit. 
In Section \ref{sec-construction_and_proof} we give the construction and proof, and we briefly discuss the result in Section \ref{sec-discussion}.
\section{Construction and proof} \label{sec-construction_and_proof}

We will start with a more general construction of a family of graphs $T_{m,t,1}$ for integers $m$ and $t$. The tree $T_{m,t,1}$ is rooted. The root (call it $v$) has $m$ children, $w_1, \ldots, w_m$. Each $w_i$ has $t$ children, $x_i^1, \ldots, x_i^t$, and each $x_i^j$ has one child, $y_i^j$. It is easy to check that $T_{m,t,1}$ has $1+m+2mt$ vertices and that it satisfies $\alpha(T_{m,t,1}) = (1+t)m$ (the unique largest independent set is $\{w_1, \ldots, w_m\} \cup \{y_i^j:1 \leq i, j \leq t\}$). See Figure \ref{fig:T43} for an example of this construction. Note that $T_{m,t,1}$ is an instance of a \emph{spherically symmetric} tree --- a rooted tree in which the number of children a vertex has depends only on its distance from the root. 

\begin{figure}[htp]
    \centering
    \includegraphics[width=9cm]{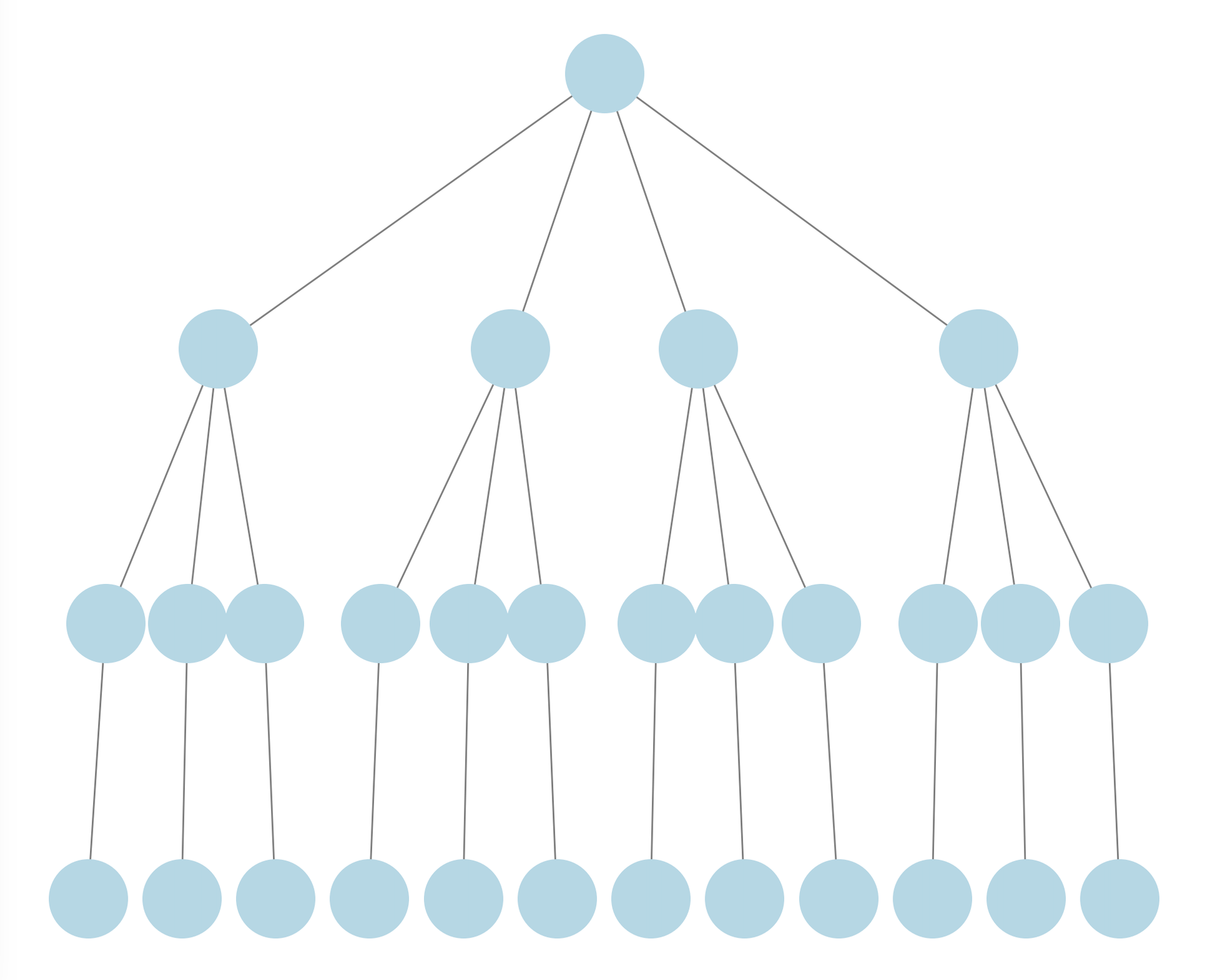}
    \caption{The tree $T_{4,3,1}$. The root is the topmost vertex.}
\label{fig:T43}
\end{figure}

In this section we prove the following theorem.  
\begin{thm} \label{thm-more-general}
For $t \leq m = m(t) \leq 2^{t/16}$ we have that for large enough $t$ the log-concavity of the independent set sequence of $T_{m,t,1}$ is broken at $mt+2$; that is,
$$
\left(i_{T_{m,t,1}}(mt+2)\right)^2 < i_{T_{m,t,1}}(mt+1)i_{T_{m,t,1}}(mt+3). 
$$
\end{thm}

Theorem \ref{thm-Theta(root(alpha))} follows immediately on taking $m=[2^{t/16}]$. 

As a simple special case of Theorem \ref{thm-more-general} we see that the tree $T_{t,t,1}$ on $1+t+t^2$ vertices has independence number $t+t^2$ and has log-concavity of the independent set sequence broken at $t^2+2$, so for large $t$ we have that log-concavity is broken at $\alpha(T_{t,t,1})-(1+o(1)\sqrt{\alpha(T_{t,t,1})})$. Computation suggests that in fact log-concavity of the independent set sequence of $T_{t,t,1}$ is broken at $t^2+2$ for all $t \geq 4$.   

\begin{proof} (Proof of Theorem \ref{thm-more-general})
Denote by $f_{m,t}(k)$ the number of independent sets in $T_{m,t,1}$ that do not include the root $v$, and by $h_{m,t}(k)$ the number that do include $v$. Observe that on deleting $v$ and all its neighbors $T_{m,t,1}$ reduces to a perfect matching on $mt$ edges, a graph that admits independent sets of size only up to $mt$. It follows that $h_{m,t}(k)=0$ for $k \geq mt+2$, and that 
\begin{equation} \label{split}
i_{T_{m,t,1}}(k) = \left\{
\begin{array}{rl}
h_{m,t}(k) + f_{m,t}(k) & \mbox{if $k \leq mt+1$} \\
f_{m,t}(k) & \mbox{if $k \geq mt+2$}.
\end{array}
\right.
\end{equation}
Further, we have $h_{m,t}(mt+1)=2^{mt}$. So our goal is to show that
$$
\left(f_{m,t}(mt+2)\right)^2 < (2^{mt} + f_{m,t}(mt+1))f_{m,t}(mt+3),
$$
which is implied by 
\begin{equation} \label{mt}
\left(f_{m,t}(mt+2)\right)^2 < 2^{mt}f_{m,t}(mt+3).
\end{equation}
We will establish \eqref{mt} by showing that for $t \leq m \leq 2^{t/16}$ we have
\begin{equation} \label{inq-f2,f3}
f_{m,t}(mt+3) \geq \binom{m}{3}2^{mt-3t}~~~\mbox{and}~~~f_{m,t}(mt+2) = \left(1+o(1)\right) \binom{m}{2}2^{mt-2t}.
\end{equation}
That \eqref{inq-f2,f3} implies \eqref{mt} for large $t$ and $m$  satisfying $m \leq 2^{t-1}$ (not just $m \leq 2^{t/16}$) is easily established.  

\smallskip

We get the lower bound on $f_{m,t}(t^2+3)$ in \eqref{inq-f2,f3} by considering those independent sets of size $mt+3$ that include exactly $3$ vertices from $\{w_1, \ldots, w_m\}$ (and so necessarily don't contain the root). There are $\binom{m}{3}$ ways to select the three vertices. Once they are chosen the remaining $mt$ vertices come from a graph that consists of $3t$ isolated vertices and $mt-3t$ isolated edges, for which there are $2^{mt-3t}$ options.

For the asymptotic determination of $f_{m,t}(mt+2)$ in \eqref{inq-f2,f3}: 
\begin{itemize}
\item Arguing exactly as we did for $f_{m,t}(mt+3)$, there are $\binom{m}{2}2^{mt-2t}$ independent sets in $T_{m,t,1}$ of size $mt+2$ that include exactly $2$ vertices from $\{w_1, \ldots, w_m\}$ (and so necessarily don't contain the root). 
\item Next, there is no contribution to $f_{m,t}(mt+2)$ from independent sets that include exactly $0$ or $1$ vertices from $\{w_1, \ldots, w_m\}$ (the maximum size of such independent sets is $mt+1$).
\item The remaining independent sets that contribute to $f_{m,t}(t^2+2)$ have $s$ vertices from $\{w_1, \ldots, w_m\}$, where $s$ satisfies $3 \leq s \leq m$ ($\binom{m}{s}$ options). The remaining $mt-(s-2)$ vertices come from a graph that consists of $st$ isolated vertices and $mt-st$ independent edges. To construct all such independent sets, we can first choose $\ell$ of the isolated edges from which not to select a vertex, then choose a vertex from each of the remaining edges, and finally choose $(s-2)-\ell$ isolated vertices to not select to complete the independent set. Here $\ell$ satisfies $0 \leq \ell \leq \min\{s-2,mt-st\}$.   
\end{itemize}
It follows from the discussion above that for each $3 \leq s \leq m$ and $0 \leq \ell \leq \min\{s-2,mt-st\}$ we have a term
$$
\binom{m}{s}\binom{mt-st}{\ell}\binom{st}{(s-2)-\ell}2^{mt-st-\ell}
$$
contributing to $f_{m,t}(mt+2)$, and that these are the only terms contributing to $f_{m,t}(mt+2)$ aside from $\binom{m}{2}2^{mt-2t}$ (which we hope to show is the dominant term). Since there at most $m^2$ such terms, we complete the verification of the asymptotic expression for $f_{m,t}(mt+2)$ in \eqref{inq-f2,f3} by showing that
$$
\frac{\binom{m}{s}\binom{mt-st}{\ell}\binom{st}{(s-2)-\ell}2^{mt-st-\ell}}{\binom{m}{2}2^{mt-2t}} = o\left(m^{-2})\right)
$$
or equiavlently
\begin{equation} \label{asy}
\frac{\binom{m}{s}\binom{mt-st}{\ell}\binom{st}{(s-2)-\ell}}{2^{(s-2)t+\ell}} = o(1).
\end{equation}
Using (here and throughout) the simple bound $\binom{n}{k} \leq n^k$ we have
\begin{eqnarray*}
\binom{m}{s} & \leq & m^s = 2^{s \log m}, \\
\binom{mt-st}{\ell} & \leq & (mt)^s \leq 2^{2s\log m}
\end{eqnarray*}
(the latter using $\ell \leq s$ and $t \leq m$) and (for $t \geq 2$, and using $s,t \leq m$)
$$
\binom{st}{(s-2)-\ell} \leq (st)^s \leq 2^{2s\log m}.
$$
Lower bounding $(s-2)t+\ell \geq st/3$ (recall $s \geq 3$) we get that the left-hand side of \eqref{asy} is upper bounded by 
$$
2^{5s\log m-(st)/3}
$$ 
As long as $m \leq 2^{t/16}$ this is $o(1)$.
 
This concludes the verification that $f_{m,t}(mt+2)=(1+o(1))\binom{m}{2}2^{mt-2t}$, and so the verification that the log-concavity of $T_{m,t,1}$ is broken at $mt+2$ for sufficiently large $t$ as long as $m \leq 2^{t/16}$.
\end{proof}

\section{Discussion} \label{sec-discussion}

There is a heuristic that explains why we might expect the log-concavity of the independent set sequence of $T_{m,t,1}$ to be broken at $mt+2$, but not at any subsequent point. Recall from \eqref{split} that for $k \geq mt+2$ we have $i_k(T_{m,t,1})=h_{m,t}(k)$. But $h_{m,t}(k)$ is the number of independent sets of size $k$ in the graph $T_{m,t,1}-v$, which consists of $m$ disjoint copies of a tree $S_{t,2}$, which in turn consists of $t$ paths of length $2$ all sharing a common start vertex. We observe (see Lemma \ref{lemma-mod-lc} below) that the independent set sequence of $S_{t,2}$ is log-concave for all $t$, and so (by the fact that the convolution of log-concave sequences is log-concave, see e.g. \cite{B2016}), the independent set sequence of $T_{m,t,1}-v$ is log-concave. So certainly log-concavity for the independent set sequence of $T_{m,t,1}$ is not broken at $mt+3$ or later. However, there is a chance that it might be broken at $mt+2$: we have $f_{m,t}(mt+2)^2 \geq f_{m,t}(mt+1)f_{m,t}(mt+3)$, but it may well be the case (and indeed is, as we have shown) that the boost given to the right-hand side by the addition of $h_{m,t}(mt+1)$ to $f_{m,t}(mt+1)$ results in the reversed inequality $f_{m,t}(mt+2)^2 < \left(f_{m,t}(mt+1)+h_{m,t}(mt+1)\right)f_{m,t}(mt+3)$.

\medskip

We do not know how far further beyond $\Theta\left(\alpha(T)/\log \alpha(T)\right)$ we can push the point where log-concavity can be broken.

\begin{question} \label{q-dist}
Is there a $c>0$ such that there are trees $T$ with arbitrarily large independence number $\alpha(T)$, for which the log-concavity of the independent set sequence of $T$ is broken at no later than around $(1-c)\alpha(T)$?
\end{question}
It was shown by Levit and Mandrescu \cite{LM2006} that the independent set sequence of a tree $T$ is monotone decreasing from $\lceil (2\alpha(T)-1)/3\rceil$ on. So in order to use points at which log-concavity is broken as an avenue to producing a tree with non-unimodal independent set sequence (and so answer Question \ref{question-AMSE} in the negative), it would be necessary to work with examples that answer Question \ref{q-dist} in the affirmative with $c > 1/3$.   

All of our computations show that $T_{m,t,1}$ either has log-concave independent set sequence or that log-concavity is broken at one point only, namely at $mt+2$. Furthermore, all of the counterexamples to log-concavity of the independent set sequence of a tree found in \cite{KL2023, KLYM2023, RS2025} have log-concavity broken at one point only. After reading an earlier version of this paper \cite{G2025}, Ferenc Bencs observed computationally that deeper spherically symmetric trees can exhibit non-log-concavity of the independent set sequence at more than one place \cite{B}. Indeed, denote by $T(2^m1^n)$ the spherically symmetric tree in which all vertices at distances $0$ through $m-1$ from the root have two children, and all vertices at distance $m$ through $m+n-1$ have one child. A {\tt Mathematica} calculation shows that, for example
\begin{itemize}
\item log-concavity of the independent set sequence of $T(2^41^9)$ is broken at two places;
\item for $T(2^51^{15})$ it is broken at three places;
\item for $T(2^61^{17})$ it is broken at eight places;
\item for $T(2^71^{23})$ it is broken at sixteen places; and
\item for $T(2^81^{27})$ it is broken at twenty-four places.
\end{itemize}

\begin{question}
For each $k \geq 1$ is there an $m=m(k)$ and $n=n(k)$ such that the log-concavity of the independent set sequence of $T(2^m1^n)$ is broken at more than $k$ places?
\end{question}

In \cite{BR2025} Bautista-Ramos used a different family of trees to verify that indeed for every $k \geq 1$ there is a tree $T_k$ such that the log-concavity of the independent set sequence of $T_k$ is broken at more than $k$ places.   

\medskip

We end with the promised verification of the log-concavity of the independent set sequence of $S_{t,2}$.

\begin{lemma} \label{lemma-mod-lc}
For all $t$, the independent set sequence of $S_{t,2}$ is log-concave.
\end{lemma}

\begin{proof}
The independence number of $S_{t,2}$ is $t+1$. For $0 \leq k \leq t+1$ the number of independent sets of size $k$ in $S_{t,2}$ is
$$
\binom{t}{k-1} + 2^k\binom{t}{k},
$$
the first term coming from independent sets that include the root of $S_{t,2}$ (the vertex in common to all the paths of length $2$) --- in this case the remaining $k-1$ vertices must be selected from $t$ isolated vertices --- and the second term coming from independent sets that don't include the root --- in this case all $k$ vertices must be selected from $t$ independent edges. So our task is to establish that for $1 \leq k \leq t$,
\begin{equation} \label{lc-star2}
\left(\binom{t}{k-1} + 2^k\binom{t}{k}\right)^2  \leq \left(\binom{t}{k-2} + 2^{k-1}\binom{t}{k-1}\right)\left(\binom{t}{k} + 2^{k+1}\binom{t}{k+1}\right).
\end{equation}
Expanding out and using
$$
\binom{t}{k}^2-\binom{t}{k-1}\binom{t}{k+1} = \frac{1}{k+1}\binom{t}{k}\binom{t+1}{k}
$$
and the equivalent identity with $k$ everywhere replaced with $k-1$, after some rearranging \eqref{lc-star2} is seen to be implied by 
$$
\frac{2^k}{k+1}\binom{t}{k}\binom{t+1}{k} \geq 2\binom{t}{k-2}\binom{t}{k+1},
$$
which in turn is equivalent to
\begin{equation} \label{lc-star2b}
\frac{2^k(t+1)}{k} \geq \frac{2(k-1)(t-k)}{t-k+2}.
\end{equation}
We have $2^k/k \geq 2$ for $k \geq 1$, so \eqref{lc-star2b} is implied by $t+1 \geq (k-1)(t-k)/(t-k+2)$. 
This is evident for $k=t$, and for $k < t$ it is implied by $t+1 \geq k-1$, 
which is true for $k \leq t$.
\end{proof}


\begin{thebibliography}{99}

\bibitem{AMSE1987}
Yousef Alavi, Paresh J. Malde, Allen J. Schwenk and Paul Erd\H{o}s, The vertex independence sequence of a graph is not constrained, {\it Congressus Numerantium} {\bf 58} (1987), 15--23.

\bibitem{BES2018}
Patrick Bahls, Bailey Ethridge and Levente Szabo, Unimodality of the independence polynomials of non-regular caterpillars, {\it Australasian Journal of Combinatorics} {\bf 71} (2018), 104--112.

\bibitem{BG2021}
Abdul Basit and David Galvin, On the independent set sequence of a tree, {\it Electronic Journal of Combinatorics} {\bf 28} (2021), article P3.23.

\bibitem{BR2025}
C\'esar Bautista-Ramos, Multiple breaks of log-concavity in the independence polynomials of trees, arXiv:2511.00334.

\bibitem{BRGGGS2019}
C\'esar Bautista-Ramos, Carlos Guill\'en-Galv\'an and Paulino G\'omez-Salgado, Log-concavity of some independence polynomials via a partial ordering, {\it Discrete Mathematics} {\bf 342} (2019), 18--28.

\bibitem{B2018}
Ferenc Bencs, On trees with real-rooted independence polynomial, {\it Discrete Mathematics} {\bf 341} (2018), 3321--3330.

\bibitem{B}
Ferenc Bencs, personal communication to the author.

\bibitem{B2016}
Mikl\'os B\'ona, Introduction to enumerative and analytic combinatorics (2nd edition), CRC Press, Boca Raton, 2016.

\bibitem{B2015}
Petter Br\"and\'en, Unimodality, log-concavity, real-rootedness and beyond, in Handbook of enumerative combinatorics, CRC Press, Boca Raton, FL, 2015, 437--483.

\bibitem{B1994}
Francesco Brenti, Log-concave and Unimodal sequences in Algebra, Combinatorics, and Geometry: an update, {\it Contemporary Mathematics} {\bf 178} (1994), 71--89.

\bibitem{CEWW2024}
Fran\c{c}ois Charton, Jordan S. Ellenberg, Adam Zsolt Wagner and Geordie Williamson, 
PatternBoost: Constructions in Mathematics with a Little Help from AI, arXiv:2411.00566 (2024). 

\bibitem{CS2007}
Maria Chudnovsky and Paul Seymour, The roots of the independence polynomial of a clawfree graph, {\it Journal of Combinatorial Theory Series B} {\bf 97} (2007), 350--357.

\bibitem{G2025}
David Galvin, Trees with non log-concave independent set sequences, arXiv:2502.10654v1 (2025).

\bibitem{GH2018}
David Galvin and Justin Hilyard, The independent set sequence of some families of trees, {\it Australasian Journal of Combinatorics} {\bf 70} (2018), 236--252.

\bibitem{H1990}
Yahya Ould Hamidoune, On the numbers of independent $k$-sets in a clawfree graph, {\it Journal of Combinatorial Theory Series B} {\bf 50} (1990), 241--244.

\bibitem{H2020}
Steven Heilman, Independent Sets of Random Trees and of Sparse Random Graphs, arXiv:2006.04756 (2020)

\bibitem{HL1972}
Ole J. Heilmann and Elliot H. Lieb, Theory of Monomer-Dimer Systems, {\it Communications in Mathematical Physics} {\bf 25} (1972), 190--233. 

\bibitem{KL2023}
Ohr Kadrawi and Vadim E. Levit, The independence polynomial of trees is not always log-concave starting from order 26, arXiv:2305.01784 (2023)

\bibitem{KLYM2023}
Ohr Kadrawi, Vadim E. Levit, Ron Yosef and Matan Mizrachi, On Computing of Independence Polynomials of Trees, in Recent Research in Polynomials (edited by Faruk \"Ozger), IntechOpen, 2023.

\bibitem{LM2002}
Vadim E. Levit and Eugen Mandrescu, On well-covered trees with unimodal independence polynomials, Proceedings of the Thirty-third Southeastern International Conference on Combinatorics, Graph Theory and Computing (Boca Raton, FL, 2002), {\it Congressus Numerantium} {\bf 159} (2002), 193--202.

\bibitem{LM2006}
Vadim E. Levit and Eugen Mandrescu, Partial unimodality for independence polynomials of
K\"onig-Egerv\'ary graphs, {\it Congressus Numerantium} {\bf 179} (2006), 109--119.

\bibitem{R}
A.J. Radcliffe, personal communication to the author.

\bibitem{RS2025}
Eric Ramos and Sunny Sun, An AI enhanced approach to the tree unimodality conjecture, arXiv:2510.18826 (2025).

\bibitem{S1989}
Richard P. Stanley, Log-concave and Unimodal sequences in Algebra, Combinatorics and Geometry, {\it Annals of the New York Academy of Sciences} {\bf 576} (1989), 500--535.

\bibitem{WZ2011}
Yi Wang and Bao-Xuan Zhu, On the unimodality of independence polynomials of some graphs, {\it European Journal of Combinatorics} {\bf 32} (2011), 10--20.

\bibitem{Z2007}
Zhi-Feng Zhu, The unimodality of independence polynomials of some graphs, {\it Australasian Journal of Combinatorics} {\bf 38} (2007), 27--33.

\end{thebibliography}
\end{document}